\documentclass[12pt,reqno]{amsart}
\usepackage{amscd,amsfonts,amsmath,amssymb,amsthm}
\usepackage{enumerate,hyperref,perpage,url}

\usepackage[margin=2.2cm, a4paper]{geometry}

\theoremstyle{plain}
\newtheorem{thm}{Theorem}

\newtheorem{cor}[thm]{Corollary}

\newtheorem{lemma}[thm]{Lemma}
\newtheorem{pblm}[thm]{Problem}

\theoremstyle{remark}
\newtheorem*{ex}{\textbf{Example}}
\newtheorem{rmk}[thm]{\textbf{Remark}}

\numberwithin{equation}{section}

\MakePerPage{footnote} \allowdisplaybreaks 

\newcommand\diag{\text{diag}}
\newcommand\dist{\text{dist}}

\newcommand\M{\mathbb M}
\newcommand\R{\mathbb R}
\newcommand\SH{\mathbb{SH}}
\newcommand\s{\mathbb{S}}

\newcommand\VERT{|||}

\title[Spherical harmonics with maximal $L^p$ $(2<p\le6)$ norm growth]{Spherical harmonics with maximal $L^p$ $(2<p\le6)$ norm growth}

\author{Xiaolong Han}
\address{Department of Mathematics, Australian National University, Canberra, ACT 0200, Australia}
\email{Xiaolong.Han@anu.edu.au}



\thanks{Research is partially supported by the Australian Research Council through Discovery Project DP120102019.} 

\begin{document}
\maketitle

\begin{abstract}
In this paper, we show that there exists a positive density subsequence of orthonormal spherical harmonics which achieves the maximal $L^p$ norm growth for $2<p\le6$, therefore giving an example of a Riemannian surface supporting such subsequence of eigenfunctions. This answers the question proposed by Sogge and Zelditch \cite{SZ2}. Furthermore, we provide an explicit lower bound on the density in this example.
\end{abstract}

\section{Introduction}

In an $n$-dimensional smooth and compact Riemannian manifold $(\M,g)$ without boundary, denote $\Delta_g$ the Laplace-Beltrami operator associated with $g$. Then $-\Delta=-\Delta_g$ has eigenvalues $0<\lambda_0^2\le\lambda_1^2\le\lambda_2^2\le\cdots$ counting multiplicities, and the eigenfunctions $\{u_j\}_{j=1}^\infty$ satisfy:
$$-\Delta u_j=\lambda^2_ju_j,$$ 

One can form an orthonormal basis $\{u_j\}_{j=1}^\infty$ in $L^2(\M)$. In this paper, we use $\|\cdot\|_p$ to mean $\|\cdot\|_{L^p(\M)}$. It was proved by Sogge \cite{So1, So2} that $\|u_j\|_p\lesssim\lambda_j^{\sigma(p)}$ with
\begin{equation}\label{eq:S}
\sigma(p)=\begin{cases}
\frac{n-1}{2}\left(\frac12-\frac1p\right) & 2<p\le\frac{2(n+1)}{n-1}=p_n,\\
n\left(\frac12-\frac1p\right)-\frac12 & p_n\le p\le\infty.
\end{cases}
\end{equation}

On the Riemannian surface, $p_2=6$. \eqref{eq:S} is sharp on the sphere $\s^2\subset\R^3$ with round metric, and saturated by certain spherical harmonics. We now introduce the terminology: The spherical harmonics are the homogeneous harmonic polynomials restricted on the sphere; we use $\SH_k$ to denote the set of such functions with homogeneous degree $k$. For each $u\in\SH_k$, $u$ is an eigenfunction of $-\Delta$ on $\s^2$ with eigenvalue $k(k+1)$:
$$-\Delta u=k(k+1)u.$$

Then the eigenfrequency $\lambda=\sqrt{k(k+1)}\approx k$; the multiplicity of the eigenvalue $k(k+1)$ is the dimension of its eigenspace, $\dim\SH_k=2k+1$; and $L^2(\s^2)=\oplus_{k=0}^\infty\SH_k$. See e.g. \cite[\S16]{Z2}. In particular, if we use the longitudinal coordinate $\phi\in[0,\pi]$ and latitudinal coordinates $\theta\in[0,2\pi)$ so that $\s^2\ni x=(\sin\phi\cos\theta,\sin\phi\sin\theta,\cos\phi)$, then one can write the standard orthonormal basis of $\SH_k$ as $\{Y_m^k\}_{m=-k}^k$:
$$Y_m^k(\phi,\theta)=C_{k,m}P_k^m(\cos\phi)e^{im\theta},$$
in which $C_{k,m}$ is the $L^2$ normalisation factor, and $P_k^m$ is the associated Legendre polynomial. The standard spherical harmonics can be categorized as
\begin{enumerate}[(i).]
\item $m=0$: $Z_k=Y_0^k$ are called zonal harmonics;
\item $-k<m<k$: $Y_m^k$ are called tesseral harmonics;
\item $m=\pm k$: $Q_{\pm k}=Y^k_{\pm k}$ are called sectoral harmonics or Gaussian beams.
\end{enumerate}

The Gaussian beams concentrate around the equator $\phi=\pi/2$ and achieve the maximal $L^p$ norms for $2<p\le6$:
\begin{equation}\label{eq:Qpnorm}
\|Q_{\pm k}\|_p=C_pk^{\sigma(p)}\approx k^{(1/2-1/p)/2},\quad2<p\le6,
\end{equation}
where $C_p$ is an absolute constant depending only on the sphere. While zonal harmonics concentrate around the north pole $\phi=0$ and the south pole $\phi=\pi$, and achieve the maximal $L^p$ norms for $p\ge6$:
\begin{equation}\label{eq:Zpnorm}
\|Z_k\|_p\approx k^{2(1/2-1/p)-1/2},\quad6\le p\le\infty.
\end{equation}

Since $\Delta$ is rotational invariant, rotating Gaussian beams or zonal harmonics generates $L^p$ norm maximizers as eigenfunctions concentrating on any great circle or any point.

However, Sogge's $L^p$ estimate \eqref{eq:S} is not sharp in many manifolds. In fact, \eqref{eq:S} is very rarely sharp, e.g. Sogge and Zelditch \cite[Theorem 1.4]{SZ1} proved that for a generic Riemannian metric $g$ on $\M$, $\|u\|_\infty=o\left(\lambda^{(n-1)/2}\right)$. Even on the sphere where \eqref{eq:S} is sharp, the maximizers (Gaussian beams and zonal harmonics) are quite sparse in the standard basis.

It is natural to ask whether the maximal $L^p$ norm estimates can be achieved in a manifold; and if so, whether the maximizers resemble Gaussian beams (concentrating on closed geodesics) or zonal harmonics (concentrating on self focal points) on the sphere. These questions have been studied by a lot of mathematicians. We refer to Zelditch's survey articles \cite{Z1, Z2} for general references and recent development in this area.

Following the notation in \cite{SZ1}, we say that $(\M,g)$ has maximal $L^p$ growth, if there exists an orthonormal eigenfunction basis $\{u_j\}\subset L^2(\M)$ such that
$$\limsup_{j\to\infty}\|u_j\|_p=\Omega\left(\lambda_j^{\sigma(p)}\right).$$

Here, $\Omega\left(\lambda_j^{\sigma(p)}\right)$ means $O\left(\lambda_j^{\sigma(p)}\right)\text{ but not }o\left(\lambda_j^{\sigma(p)}\right)$. Therefore, there exists a subsequence $\{u_{j_m}\}$ such that
$$\lim_{m\to\infty}\frac{\|u_{j_m}\|_p}{\lambda_{j_m}^{\sigma(p)}}\ge c_\M$$
for some constant $c_\M$ depending only on $(\M,g)$. 

If $p_n\le p\le\infty$, Sogge, Toth, and Zelditch proved some conditions involving ``blow-down'' points for the manifold to have maximal $L^p$ growth. This in some sense resembles zonal harmonics with maximal $L^p$ ($p_n\le p\le\infty$) norm growth at a point, in which case all the geodesics emitting from this point return back at the same time. We refer to their papers \cite{SZ1, STZ} for details and other related results.

If $2<p\le p_n$, fewer results are known. In the following, we concentrate on the case of Riemannian surfaces. Sogge \cite{So3} (see also Bourgain \cite{B2}) proved that $\|u_j\|_p=o\left(\lambda_j^{\sigma(p)}\right)$ for $2<p<6$ and all eigenfunctions if and only if 
$$\sup_{\gamma\in\Pi}\int_{\dist_g(\gamma,y)\le\lambda_j^{-1/2}}|u_j(y)|^2dy=o(1),$$
where $\Pi$ is the space of all unit-length geodesics in $\M$, and $\dist_g$ is the geodesic distance associated with $g$. (This result has been recently extended to higher dimensions by Blair and Sogge \cite{BS}.)

Eigenfunction estimates are intimately related to the dynamical properties of geodesic flows. If $(\M,g)$ has nonpositive curvature, then the geodesic flow on $\M$ is automatically ergodic. Sogge and Zelditch \cite[Corollary 1.2]{SZ3} proved that in this case, for $\varepsilon>0$ and $2<p<6$ are fixed, there is a $\lambda(\varepsilon,p)<\infty$ so that
\begin{equation}\label{eq:SZ3}
\|u\|_p\le\varepsilon\lambda^{\sigma(p)}\|u\|_2,\quad\text{for all}\quad\lambda>\lambda(\varepsilon,p).
\end{equation}

Duistermaat-Guillemin's small-$o$ improvement on the remainder term in Weyl law expansion \cite{DG} provides a natural condition of the periodic geodesics. Sogge and Zelditch \cite[Theorem 2.1]{SZ2} proved that 
\begin{thm}[Sogge and Zelditch \cite{SZ2}]\label{thm:SZ}
If the measure of periodic geodesics is zero, then 
\begin{eqnarray}\label{eq:SZ}
&&\text{for any orthonormal eigenfunction basis $\{u_j\}$, one can find a full density}\nonumber\\
&&\text{subsequence $\{u_{j_m}\}$ such that $\|u_{j_m}\|_p=o\left(\lambda_{j_m}^{\sigma(p)}\right)$ for $2<p<6$.}
\end{eqnarray}
\end{thm} 

Here, we define the density $D$ of a subsequence $\{j_m\}\subset\mathbb N$ as
$$D=\lim_{N\to\infty}\frac{\#\{j_m\le N\}}{N}.$$

When $D=0$ ($>0$ or $=1$), we call such subsequence a zero (positive or full) density subsequence. 

On $\s^2$, Gaussian beams maximize the $L^p$ ($2<p\le6$) norms. However, they are very sparse among the standard spherical harmonics basis: In each $\SH_k$, there are only two Gaussian beams $Q_{\pm k}$, then 
\begin{equation}\label{eq:Qdensity}
\frac{2N}{\sum_{k=0}^{N-1}2k+1}=O\left(\frac1N\right)\to0,\quad\text{as}\quad N\to\infty,
\end{equation}
which means that the subsequence of Gaussian beams is of density zero, while other standard spherical harmonics $Y_m^k$ ($-k<m<k$), with smaller $L^p$ ($2<p\le6$) norms, constitute the ``majority'' of this basis. 

When $p=4$, additional technical tools are available. Sogge and Zelditch \cite[Theorem 3.1]{SZ2} were able to compute the average $L^4$ norm of the standard spherical harmonics in $\SH_k$:
$$\frac{1}{2k+1}\sum_{m=-k}^m\int_{\s^2}|Y^k_m|^4\approx\log k,\quad k\to\infty.$$

Note that $\|Q_{\pm k}\|_4\approx k^{1/8}$. This implies that among the standard spherical harmonics, most of them have very slow $L^4$ norm growth. Furthermore, they \cite[\S5]{SZ2} conjectured and sketched a proof that, for the random orthonormal eigenfunction bases, the expectation of the average $L^4$ norm is bounded by a constant.

These results lead to the question: Can one find an orthonormal eigenfunction basis in $\SH_k$ with average $L^4$ norm growing faster than logarithmically, e.g. polynomially, as $k\to\infty$? In fact, the existence of a positive density subsequence with polynomial $L^4$ norm growth implies that the average also grows polynomially. However, it was unclear that the answer would be true for any Riemannian surface; and they proposed the following stronger version of this question.

\begin{pblm}[Sogge and Zelditch \cite{SZ2}]\label{pblm:SZ}
For which Riemannian surface (if any) does there exist an orthonormal basis of eigenfunctions $\{u_j\}$ for which there exists a positive density subsequence $\{u_{j_m}\}$ so that it achieves the maximal $L^4$ norm growth, i.e., $\|u_{j_m}\|_4=\Omega\left(\lambda_{j_m}^{1/8}\right)$? Or is \eqref{eq:SZ} in Theorem \ref{thm:SZ} valid on any compact Riemannian surface?
\end{pblm}

In this paper, we answer the first question positively, therefore provide the example that \eqref{eq:SZ} can not be valid on arbitrary Riemannian surface without any condition on the geodesics. In fact, our theorem, for a positive density subsequence of orthonormal spherical harmonics, extends the answer to all $2<p\le 6$; and a density bound is provided.

\begin{thm}\label{thm:main}
There exists a constant $0<D<1$ that, for all $k\ge k_0$, if $m=\lfloor D(2k+1)\rfloor$, the largest integer not greater than $D(2k+1)$, then one can find an orthonormal set $\{u_i\}_{i=1}^{m}\subset\SH_k$ such that $\|u_i\|_p\ge\frac12C_pk^{\sigma(p)}$ for all $i$ and $2<p\le 6$. 

Recall that $C_pk^{\sigma(p)}$ is the $L^p$ norm of a Gaussian beam $Q_k\in\SH_k$. Moreover, one can choose $D$ such that
\begin{equation}\label{eq:D}
(7+1296D)c_0^{-1/D}\le\frac{1}{25},
\end{equation}
where $c_0=e^{1/72}>1$. For example, $D=1/400=0.25\%$ will do.
\end{thm}

We remark that \eqref{eq:D} in the theorem is not optimal, but achieving this non-optimal density condition requires careful work in the proof. 

If we complete the orthonormal set $\{u_i\}_{i=1}^m\subset\SH_k$ to an orthonormal basis, then evidently the density of $\cup_k\{u_i\}_{i=1}^{m}$ is $D$ in this basis since $\dim\SH_k=2k+1$ and $D(2k+1)-1<m\le D(2k+1)$. Thus when $k$ is large enough,
$$\frac{1}{2k+1}\sum_{i=1}^{2k+1}\|u_i\|_p\ge\frac{1}{2k+1}\sum_{i=1}^m\|u_i\|_p\ge\frac{\lfloor D(2k+1)\rfloor}{2k+1}\cdot\frac 12C_pk^{\sigma(p)}\ge\frac D3\cdot C_pk^{\sigma(p)}.$$

We deduce the following result.

\begin{cor}
Using the same notations and constants as in Theorem \ref{thm:main}, there exists an orthonormal basis $\{u_i\}_{i=1}^{2k+1}$ in $\SH_k$ such that
$$\frac{1}{2k+1}\sum_{i=1}^{2k+1}\|u_i\|_p\ge\frac D3\cdot C_pk^{\sigma(p)},\quad k\to\infty,$$
for $2<p\le 6$.
\end{cor}

For $p>6$, one can ask a similar question as Problem \ref{pblm:SZ}:
\begin{pblm}\label{pblm:largep}
For which Riemannian surface (if any) does there exist an orthonormal basis of eigenfunctions $\{u_j\}$ for which there exists a positive density subsequence $\{u_{j_m}\}$ so that it achieves the maximal $L^p$ norm growth for some $p>6$?
\end{pblm}

To my knowledge, this is unknown now. Following the argument in this paper, we are able to prove a weaker version:
\begin{cor}\label{cor:largep}
Using the same notations and constants as in Theorem \ref{thm:main}, there exists an orthonormal spherical harmonics set $\{u_i\}_{i=1}^m\subset\SH_k$  with positive density $D$ satisfying \eqref{eq:D} such that
$$\|u_i\|_p\approx k^{\frac14-\frac{1}{2p}},$$
for $p>6$. In particular, $\|u_i\|_\infty\approx k^{1/4}$.
\end{cor}

Here, recall that for Gaussian beams, $\|Q_k\|_p\approx k^{1/4-1/(2p)}$ in \eqref{eq:Qpnorm}, comparing to the maximal $L^p$ norm growth for $p>6$ of zonal harmonics $\|Z_k\|_p\approx k^{1/2-2/p}$.

\subsection*{Further remarks and some background}
Since $\|u\|_p\gtrsim\|u\|_2$ for $p>2$ on compact manifolds, it is also interesting to ask whether one can achieve the minimal $L^p$ growth of orthonormal eigenfunctions on $(\M,g)$. Toth and Zelditch \cite{TZ} proved that under certain condition, if all the eigenfunctions have minimal $L^\infty$ norm growth $O(1)$, then the manifold must be a flat torus. See \cite{Z1, Z2} for more information. 

On $\s^2$, VanderKam \cite[Theorem 1.2]{V} proved that for almost all orthonormal eigenfunction bases, the $L^\infty$ norm of eigenfunctions is bounded by $\log^2\lambda$. But it is unknown whether there exists a uniformly bounded orthonormal eigenfunction basis or a uniformly bounded subsequence of eigenfunctions with some density in an orthonormal basis.

On $\s^3$, regarded as the boundary of the unit ball in $\mathbb C^2$, Bourgain \cite{B1} constructed a uniformly bounded orthonormal basis for the Hilbert space of holomorphic polynomials. It is still open whether a bounded basis exists in higher dimensions, i.e. on $\s^{2n-1}\subset\mathbb C^n$ for $n\ge3$, or more generally on the boundary of a relatively compact strictly pseudoconvex domain in a complex manifold. Very recently, Shiffman \cite{Sh} proved a partial result that one can select a positive density sequence of uniformly bounded orthonormal sections of a holomorphic bundle over a compact K\"ahler manifold.

In the opposite direction of Problems \ref{pblm:SZ} and \ref{pblm:largep}, if a manifold $(\M,g)$ supports a positive density subsequence of orthonormal eigenfunctions achieving maximal $L^p$ norm growth for some $p>2$, which is a much stronger condition than $(\M,g)$ having maximal $L^p$ growth, then in this case what can we say about the dynamical properties of its geodesic flow? One can weaken the question by assuming the density is of order $O(1/N)$ (as the density of Gaussian beams \eqref{eq:Qdensity} in the standard spherical harmonics basis). In dimension two, Theorem \ref{thm:SZ} from \cite{SZ2} says that if the measure of periodic geodesic is zero, then there does not exist a positive density subsequence of eigenfunctions achieving maximal $L^p$ norm growth for $2<p<6$. While \eqref{eq:SZ3} from \cite{SZ3} gives a deterministic statement that all eigenfunctions can not achieve maximal $L^p$ norm growth for $2<p<6$ if the curvatures of the manifold are nonpositive. Results for $p>6$ can be found in \cite{SZ1, STZ, TZ} and references therein.\\

After this paper was completed, I was informed about Bourgain's unpublished note: In \cite{B3}, Bourgain proved the existence of a positive density subsequence of spherical harmonics on $\s^2$ with maximal $L^4$ norm growth in Theorem \ref{thm:main}, therefore answers Problem \ref{pblm:SZ}. His result can be generalized to $L^p$ for all $2<p\le6$. However, his approach is different with the present one, and no explicit density bound is given therein.\\

Throughout this paper, $A\lesssim B$ ($A\gtrsim B$) means $A\le cB$ ($A\ge cB$) for some constant $c$ depending only on the sphere; $A\approx B$ means $A\lesssim B$ and $B\lesssim A$; the constants $c$ and $C$ may vary from line to line.

\subsection*{Organization of the paper}
In Section \ref{sec:pre}, we outline the proof of the main theorem, Theorem \ref{thm:main}, and recall the crucial tools used after: diagonally dominant matrices and Wigner D-matrix; in Section \ref{sec:main}, we give the proofs of Theorem \ref{thm:main} and Corollary \ref{cor:largep}; in Section \ref{sec:loc}, we give some characterization on the $L^2$ localization of the spherical harmonics in Theorem \ref{thm:main}.

\section{Outline of the proof and preliminaries}\label{sec:pre}

\subsection*{Outline of the proof}
In $\SH_k$ with dimension $2k+1$, to construct an orthonormal set of spherical harmonics with large $L^p$ ($2<p\le6$) norms, one of course wants to choose Gaussian beams since they are the maximizers. For $0<D<1$ and $m=\lfloor D(2k+1)\rfloor<2k+1$, if we choose $m$ Gaussian beams $\{q_i\}_{i=1}^m$ concentrating on well-separated great circles, they are $L^p$ norm maximizers, however not orthogonal to each other in $L^2(\s^2)$. In order to meet the orthogonality condition, one can use Gram-Schmidt process to orthonormalize them, and then compute how many of the resulting eigenfunctions can still maximize $L^p$ norms. This is in fact the approach suggested in \cite{SZ2}. However, the computation is too complicated, and it is extremely difficult to estimate the $L^p$ norms of the new spherical harmonics after multiple steps in the process\footnote{Bourgain's argument \cite{B3}, however, uses Gram-Schmidt process together with some ``Hilbertian system'' tools to construct a new orthonormal system from a set of Gaussian beams.}.

Instead of using Gram-Schmidt process, where one modifies $\{q_i\}_{i=1}^m$ inductively to make the following one normalized and orthogonal to all the preceding ones, we modify the Gaussian beams $\{q_i\}_{i=1}^m$ all together and find a new set of orthogonal eigenfunctions $\{u_i\}_{i=1}^m$. Each $u_i$ is close to $q_i$ in both $L^2$ and $L^p$ norms for $p>2$. This in turn is reduced to the functional analysis of the corresponding matrices $\big(\langle q_i,q_j\rangle\big)$ and $\big(\langle u_i,u_j\rangle\big)$, and can be done if we choose the density $D$ small enough.

\subsection*{Diagonally dominant matrices}
The main step in proving Theorem \ref{thm:main} to show that the two matrices $\big(\langle q_i,q_j\rangle\big)$ and $\big(\langle u_i,u_j\rangle\big)$ are both strictly diagonally dominant; therefore it is necessary to recall some standard facts about this type of matrix. For more information, see \cite[Chapter 6]{HJ}.

Let $A=(a_{ij})$ be a $n\times n$ complex matrix, i.e. $A\in M_n(\mathbb C)$, $A$ is said to be diagonally dominant if 
$$|a_{ii}|\ge\sum_{j=1,j\ne i}^n|a_{ij}|=R'_i(A)$$
for all $i=1,...,n$. And it is said to be strictly diagonally dominant if all the above inequalities are strict for $i=1,...,n$.
Here, $R_i'(A)$ is called the $i$-th deleted absolute row sums of $A$. The following theorem states that all the eigenvalues of $A$ fall into the ``Ger\v sgorin discs'' involving the diagonal entries and the deleted absolute row sums.

\begin{lemma}[Ger\v sgorin disc theorem]\label{thm:G}
All the eigenvalues of $A$ are located in the union of $n$ Ger\v sgorin discs:
$$\bigcup_{i=1}^n\big\{z\in\mathbb C:|z-a_{ii}|\le R'_i(A)\big\}.$$
\end{lemma}

One can regard $A$ as a perturbation of the diagonal matrix $\diag(a_{11},...,a_{nn})$ by adding $R'_i(A)$ to the $i$-th row, then Lemma \ref{thm:G} concludes that the eigenvalues of $A$ are in the neighborhoods of $a_{11},...,a_{nn}$. See \cite[Theorem 6.1.1]{HJ} for a detailed proof. 

\subsection*{Inner product of two Gaussian beams}

Recall that the two Gaussian beams $Q_{\pm k}$ defined in the introduction are
$$Q_{\pm k}(\phi,\theta)=C_kP_k^k(\cos\phi)e^{\pm ik\theta}.$$

Both of them concentrate around the equator $G=\{\phi=\pi/2\}$, the great circle perpendicular to the north pole $\phi=0$, and decrease in $\phi$ exponentially away from $G$. However, they propagate in opposite directions on $G$; by the right-hand rule, we say that $Q_k$ has pole as the north pole and propagates in the positive direction on $G$. So in our terminology, $Q_{-k}$ has pole $\phi=\pi$ as the south pole. Any other Gaussian beam with pole as the north pole differs with $Q_k$ only by a phase shift, and has the form
$$Q_k^\alpha(\phi,\theta)=C_kP_k^k(\cos\phi)e^{ik(\theta+\alpha)},$$
where $\alpha\in[-\pi/k,\pi/k)$, and $ik\alpha$ is the initial phase. This is equivalent to the statement that
$$Q_k^\alpha=e^{ik\alpha}Q_k$$
on $\s^2$.

Given two Gaussian beams $q_1$ and $q_2$ with respect to poles $x_1$ and $x_2$ and concentrating around great circles $G_1$ and $G_2$, we need to compute their inner product
$$\langle q_1,q_2\rangle=\int_{\s^2}q_1\overline q_2.$$

\begin{ex} Using the above notation, it is easy to see that for two Gaussian beams $Q_k$ and $Q_k^\alpha$ with the same pole as the north pole,
$$\langle Q_k^\alpha,Q_k\rangle=e^{ik\alpha}\langle Q_k,Q_k\rangle=e^{ik\alpha},$$
which is exactly the phase shift.
\end{ex}

Generally, we have the following lemma.

\begin{lemma}\label{thm:W} Admitting the above notations about $q_1$ and $q_2$,
$$\langle q_1,q_2\rangle=e^{ik\alpha}\left(\cos\frac{\beta}{2}\right)^{2k},$$ 
in which $0\le\beta\le\pi$ is the angle between the poles $x_1$ and $x_2$, and $e^{ik\alpha}$ is the phase shift at the intersecting points of $G_1$ and $G_2$. That is,
$$e^{ik\alpha}=\frac{q_1(y_1)}{q_2(y_1)}=\frac{q_1(y_2)}{q_2(y_2)}$$
where $\alpha\in[-\pi/k,\pi/k)$, and $y_1$ and $y_2$ are the two intersecting points of $G_1$ and $G_2$. If the poles $x_1$ and $x_2$ coincide, we can choose any non-vanishing point on $\s^2$ to get the phase shift.
\end{lemma}

Let us fix the cartesian coordinates $(x^{(1)},x^{(2)},x^{(3)})$ and the corresponding polar coordinates $(\phi,\theta)$ for $x\in\s^2$:
$$\begin{cases}
x^{(1)}=\sin\phi\cos\theta,&\\
x^{(2)}=\sin\phi\sin\theta,&\\
x^{(3)}=\cos\phi.
\end{cases}$$
 
For each orthogonal matrix $R$ in $SO(3)$, the special orthogonal group in $\R^3$, it can be specified by its Euler angles $(\alpha,\beta,\gamma)$, written as $R_{\alpha\beta\gamma}$. Then $R_{\alpha\beta\gamma}$ is associated with the operator $P_{\alpha\beta\gamma}$ of 
\begin{enumerate}[(a).]
\item rotating about $x^{(3)}$-axis by an angle $\alpha$,
\item rotating about $x^{(2)}$-axis by an angle $\beta$, 
\item rotating about $x^{(1)}$-axis by an angle $\gamma$:
\end{enumerate} 
$$P_{\alpha\beta\gamma}f(y)=f(R_{\alpha\beta\gamma}y),\quad y\in\s^2,$$
for a function $f$ on the sphere. The axes remain fixed in these rotations. Here, we follow \cite[Appendix A]{W} for the convention of orientation by the right-hand rule. 

Recall that from the introduction, the standard orthonormal basis in $\SH_k$ consists of 
$$Y_m^k(\phi,\theta)=C_{k,m}P_k^m(\cos\phi)e^{im\theta},\quad m=-k,...,k.$$

For example, $R_{\alpha00}$ is the rotation about $x^{(3)}$-axis through an angle $\alpha$, then in polar coordinates $(\phi,\theta)$, $\phi$ is unchanged, and $\theta$ goes into $\theta+\alpha$. Thus
\begin{equation}\label{eq:alpha}
P_{\alpha00}Y^k_m(\phi,\theta)=e^{im\alpha}Y^k_m(\phi,\theta),
\end{equation}
and if $\alpha=2l\pi/m$, $l\in\mathbb N$, then $P_{\alpha00}$ keeps $Y_m^k$ unchanged. This is indeed the fact that $Y^k_m$ is a joint eigenfunction of the Laplace-Beltrami operator $\Delta$ and $L_3=i\partial_\theta$, the generator of rotations about $x^{(3)}$-axis. See \cite[\S16]{Z2} for a detailed discussion.

Generally, $P_{\alpha\beta\gamma}Y^k_m$, as a new function on $\s^2$, is still an eigenfunction of $-\Delta$ since $\Delta$ is invariant under rotations, (and may not be an eigenfunction of $L_3$ since it is not rotational invariant). Therefore, $P_{\alpha\beta\gamma}Y^k_m\in\SH_k$, and can be expanded by the standard spherical harmonics basis:
$$P_{\alpha\beta\gamma}Y^k_m=\sum_{m'=-k}^kD^k_{m'm}(\alpha,\beta,\gamma)Y_{m'}^k,$$
where 
$$D^k_{m'm}(\alpha,\beta,\gamma)=\langle P_{\alpha\beta\gamma}Y^k_m,Y_{m'}^k\rangle$$
since $Y_{m'}^k$'s are orthonormal. The $(2k+1)\times(2k+1)$ matrix with entries $D^k_{m'm}(\alpha,\beta,\gamma)$ is called a Wigner D-matrix. See \cite[Chapter 15]{W} for a complete treatment on this subject. In particular, Equation (15.27a) in \cite{W} is
\begin{equation}\label{eq:Dmatrix}
\langle P_{\alpha\beta\gamma}Y^k_m,Y_m^k\rangle=D^k_{mm}(\alpha,\beta,\gamma)=e^{im\alpha}\left(\cos\frac{\beta}{2}\right)^{2k}e^{im\gamma}.
\end{equation}

Now we proceed to prove Lemma \ref{thm:W} using the above tools.

\begin{proof}[Proof of Lemma \ref{thm:W}]
Because we only care about two Gaussian beams $q_1$ and $q_2$ with poles $x_1$ and $x_2$, we can assume $x_1$ is the north pole, and $x_2$ is on the $x^{(1)}x^{(3)}$-plane with polar coordinate $(\beta,0)$, $0<\beta\le\pi$. When $\beta=0$, the two poles coincide, and this case has been discussed in the beginning of this subsection. 

We aim to find a rotation $R$ such that the operator $P$ associated with $R$ satisfies
$$Pq_1=q_2.$$

Denote the two great circles perpendicular to $x_1$ and $x_2$ as $G_1$ and $G_2$. Then they intersect at two points on $x^{(2)}$-axis:
$$y_1=\left(\frac{\pi}{2},\frac{\pi}{2}\right)\quad\text{and}\quad y_2=\left(\frac{\pi}{2},\frac{3\pi}{2}\right).$$

We know that $|q_1|$ achieves its maximal value $C_kP_k^k(0)$ on $G_1$, and $|q_2|$ achieves the same maximal value on $G_2$. Therefore, at $y_1$ and $y_2$, 
$$\left|q_1(y_1)\right|=\left|q_1(y_2)\right|=\left|q_2(y_1)\right|=\left|q_1(y_2)\right|=C_kP_k^k(0).$$

Then there is a complex constant $e^{ik\alpha}$ for $\alpha\in[-\pi/k,\pi/k)$ such that
\begin{equation}\label{eq:q}
\frac{q_1(y_1)}{q_2(y_1)}=e^{ik\alpha}.
\end{equation}

We can also conclude that
$$\frac{q_1(y_2)}{q_2(y_2)}=\frac{e^{ik\pi}q_1(y_1)}{e^{ik\pi}q_2(y_1)}=e^{ik\alpha},$$
which just states that the phase changing from $y_1=(\pi/2,\pi/2)$ to $y_2=(\pi/2,3\pi/2)$ is $k\pi$ for both $q_1$ and $q_2$.

Now observe that from \eqref{eq:alpha},
$$P_{(-\alpha)(-\beta)0}q_1=P_{0(-\beta)0}P_{(-\alpha)00}q_1=e^{-ik\alpha}P_{0(-\beta)0}q_1$$
is also a Gaussian beam with pole $x_2$. Furthermore, because $y_1$ is invariant under the rotation $R_{0(-\beta)0}$ as it is on $x^{(2)}$-axis: $R_{0(-\beta)0}y_1=y_1$,
$$P_{(-\alpha)(-\beta)0}q_1(y_1)=e^{-ik\alpha}P_{0(-\beta)0}q_1(y_1)=e^{-ik\alpha}q_1\big(R_{0(-\beta)0}y_1\big)=e^{-ik\alpha}q_1(y_1)=q_2(y_1)$$
from \eqref{eq:q}. Then 
$$P_{(-\alpha)(-\beta)0}q_1=q_2$$
must hold from the fact that $P_{(-\alpha)(-\beta)0}q_1$ and $q_2$ are two Gaussian beams with the same pole and have the same phase at one point $y_1$. This implies
$$\langle q_1,q_2\rangle=\overline{\langle q_2,q_1\rangle}=\overline{\left\langle P_{(-\alpha)(-\beta)0}q_1,q_1\right\rangle}=\overline{e^{-ik\alpha}D^k_{kk}(0,\beta,0)}=e^{ik\alpha}\left(\cos\frac{\beta}{2}\right)^{2k}$$
by \eqref{eq:Dmatrix}.
\end{proof}

\begin{rmk}
In the next section, we shall select a family of Gaussian beams $\{q_i\}_{i=1}^m$ on the sphere with arbitrary phase shifts at the intersecting points. It is crucial to observe that $|\langle q_i,q_j\rangle|$ depends only on the angle between the corresponding poles.
\end{rmk}

\section{Proofs of Theorem \ref{thm:main} and Corollary \ref{cor:largep}}\label{sec:main}

Now we begin the proof of Theorem \ref{thm:main}. Follow the notations in the previous section: For $0<D<1$ to be chosen later, let $m=\lfloor D(2k+1)\rfloor<2k+1$. We choose $m$ Gaussian beams $\{q_i\}_{i=1}^m$ corresponding to $m$ roughly evenly separated poles in the following way.

Select the set of $m$ roughly evenly separated points $\{x_i\}_{i=1}^m$ on $\s^2$ with largest distance between nearby pairs, then the distance between two nearby points $d$ satisfy
$$\frac{4\pi}{\pi(2d)^2}\le m=\lfloor D(2k+1)\rfloor\le\frac{4\pi}{\pi(d/3)^2},$$
because every disc on $\s^2$ with radius $2d$ contains at least one point, and every disc with radius $d/3$ contains at most one point. Therefore, 
\begin{equation}\label{eq:d}
\frac{1}{\sqrt{\lfloor D(2k+1)\rfloor}}\le d\le\frac{6}{\sqrt{\lfloor D(2k+1)\rfloor}}.
\end{equation}

For each point $x_i$, let $q_i$ be a Gaussian beam with the pole $x_i$. (Recall the terminology of Gaussian beams and their corresponding poles in Section \ref{sec:pre}.) From Lemma \ref{thm:W},
$$|\langle q_i,q_j\rangle|=\left(\cos\frac{\beta_{ij}}{2}\right)^{2k},$$
in which $\beta_{ij}$ is the angle between the poles $x_i$ and $x_j$ ($0\le\beta_{ij}\le\pi$). Denote the $n\times n$ Hermitian matrix
$$E=\big(\langle q_i,q_j\rangle\big).$$ 

We show that if the density $D$ is chosen small enough, then $E$ is strictly diagonally dominant. In order to evaluate the deleted absolute row sums
$$R'_i(E)=\sum_{j=1,j\ne i}^m|\langle q_i,q_j\rangle|,$$
we assume $i=1$ and $x_1$ is the north pole without loss of generality. Using the coordinates $(\phi,\theta)$, we divide $\phi\in[0,\pi]$ into $\lceil\pi/d\rceil$ intervals with equal length 
$$\delta=\frac{\pi}{\lceil\pi/d\rceil}$$
where $\lceil\pi/d\rceil$ is the ceiling function, that is, the smallest integer not less than $\lceil\pi/d\rceil$: 
$$\pi/d\le\lceil\pi/d\rceil<\pi/d+1.$$

Then $d/2\le\delta\le d$, and by \eqref{eq:d}
\begin{equation}\label{eq:delta}
\frac{1}{2\sqrt{\lfloor D(2k+1)\rfloor}}\le\delta\le\frac{6}{\sqrt{\lfloor D(2k+1)\rfloor}}.
\end{equation}

Write the partition
$$\mathcal P:0<\delta<2\delta<\cdots<\lceil\pi/d\rceil\delta=\pi,$$
then $\mathcal P$ provides $\lceil\pi/d\rceil$ strips $S_l$ with $l=1,...,\lceil\pi/d\rceil$ on the sphere, which can be divided into three groups:
\begin{enumerate}[(I).]
\item $l=1$: The number of poles in $S_1$ is bounded by a constant since $\delta\le d$. They correspond to the Gaussian beams which have the largest inner product with $q_1$. 

\item $l$ is small: The strips are close to the north pole $x_1$, and they contain less poles comparing the ones with large $l$. However, the poles have relatively small angles with $x_1$, and the corresponding Gaussian beams have relatively large inner product with $q_1$. The summation in this group can be carefully examined, and then controlled if we increase the length of intervals $\delta$, which subsequently results to decreasing the density $D$.

\item $l$ is large: The strips are away from the north pole $x_1$. Even they contain more poles with number of them growing polynomially, the inner product with $q_1$ decrease exponentially when $k\to\infty$. Therefore, the summation in this group can always be well controlled, and in fact is independent of $\delta$ and $D$ if the strips are away from the north pole by a constant angle.
\end{enumerate}

Since $\cos^{2k}$ decreases exponentially away from $1$, the summations in Groups (I) and (II) above essentially contributes the majority of $R'_1(E)$. We divide $\mathcal P$ into three groups: $l=1$, $l\delta\le1/2$, and $l\delta>1/2$ according to the relation
\begin{equation}\label{eq:cos}
\cos(\beta)=1-\frac{\beta^2}{2}+O(\beta^4)\le1-\frac{\beta^2}{3},
\end{equation}
if $0\le\beta\le1/2$.

On the $l$-th strip $S_l$, the poles $x_j$ falling into the strip $S_l$ have angles with the north pole between $(l-1)\delta$ and $l\delta$. To estimate the number of poles in $S_l$, we have
\begin{itemize}
\item The number of poles in $S_1$ is bounded by 
$$\frac{2\pi\sin\delta}{d}\le7.$$

\item The surface area of $S_l$ ($l\ge2$) is
$$\int_{(l-1)\delta}^{l\delta}2\pi\sin\phi\,d\phi=4\pi\sin\frac\delta2\sin\frac{(2l-1)\delta}{2},$$
and the number of poles in $S_l$ is bounded by 
$$\frac{4\pi\sin\frac\delta2\sin\frac{(2l-1)\delta}{2}}{\pi(d/3)^2}\le\frac{18\sin\frac{(2l-1)\delta}{2}}{\delta}\le\frac{36\sin(l-1)\delta}{\delta}.$$
\end{itemize}

Compute that
\begin{eqnarray*}
&&R'_1(E)\\
&=&\sum_{j=2}^m|\langle q_1,q_j\rangle|=\sum_{l=1}^{\lceil\pi/d\rceil}\sum_{x_j\in S_l}|\langle q_1,q_j\rangle|\\
&\le&7\max_{x_j\in S_1}|\langle q_1,q_j\rangle|+\sum_{l=2}^{\lceil\pi/d\rceil}\frac{36\sin(l-1)\delta}{\delta}\,\max_{x_j\in S_l}|\langle q_1,q_j\rangle|\\
&\le&7\left(\cos\frac\delta2\right)^{2k}+36\sum_{l=2,l\delta\le1/2}(l-1)\left(\cos\frac{(l-1)\delta}{2}\right)^{2k}+36\sum_{l=2,l\delta>1/2}(l-1)\left(\cos\frac{(l-1)\delta}{2}\right)^{2k}.
\end{eqnarray*}

Note that the angle between $x_j$ and $x_1$, $\beta_{ij}\ge d\ge\delta$, and we use this estimate for the poles in $S_1$. Then
\begin{enumerate}[(I).]
\item $l=1$: The number of poles in $S_1$ is bounded by $7$. Using \eqref{eq:delta} and \eqref{eq:cos},
\begin{eqnarray*}
7\left(\cos\frac\delta2\right)^{2k}&\le&7\exp\left[2k\log\left(1-\frac{\delta^2}{12}\right)\right]\\
&\le&7\exp\left[-\frac{k\delta^2}{6}\right]\\
&\le&7\exp\left[-\frac{k}{24\lfloor D(2k+1)\rfloor}\right]\\
&\le&7\exp\left[-\frac{1}{72D}\right]\\
&=&7c_0^{-1/D}\to0,\quad\text{as}\quad D\to0,
\end{eqnarray*}
with an absolute constant $c_0=e^{1/72}>1$.

\item $l\ge2$ and $l\delta\le1/2$: Similarly as in (I),
\begin{eqnarray*}
&&36\sum_{l=2,l\delta\le1/2}(l-1)\left(\cos\frac{(l-1)\delta}{2}\right)^{2k}\\
&\le&36\sum_{l=2,l\delta\le1/2}(l-1)\exp\left[-\frac{(l-1)^2}{72D}\right]\\
&\le&36\sum_{l=1}^\infty l\cdot\exp\left[-\frac{l^2}{72D}\right]\\
&\le&36\int_1^\infty t\cdot\left[\exp\left(\frac{1}{72D}\right)\right]^{-t^2}\,dt\\
&=&1296Dc_0^{-1/D}\to0,\quad\text{as}\quad D\to0.
\end{eqnarray*}

\item $l\ge2$ and $l\delta>1/2$: We have $(l-1)\delta>1/4$, then
\begin{eqnarray*}
&&36\sum_{l=2,l\delta>1/2}(l-1)\left(\cos\frac{(l-1)\delta}{2}\right)^{2k}\\
&\le&36\left(\cos\frac18\right)^{2k}\sum_{l=2,l\delta>1/2}(l-1)\\
&\le&36\left(\cos\frac18\right)^{2k}(2k+1)^2\to0,\quad\text{as}\quad k\to\infty,
\end{eqnarray*}
which is independent of $\delta$ and $D$.
\end{enumerate}

\begin{rmk}
Comparing (I) and (II),
$$7c_0^{-1/D}\ge1296Dc_0^{-1/D}$$
if and only if
$$D\le\frac{7}{1296},$$
and this is the case when the number of the poles $m=\lfloor D(2k+1)\rfloor$ is small on the sphere. The majority of the contribution to $\sum_j|\langle q_1,q_j\rangle|$ just comes from the poles around $x_1$ in the first strip $S_1$. Away from $S_1$, the contribution to the summation is essentially negligible.
\end{rmk}

Adding the above quantities in (I), (II), and (III) together, we arrive at
\begin{equation}\label{eq:r}
R'_1(E)\le 7c_0^{-1/D}+1296Dc_0^{-1/D}+o(k)=(7+1296D)c_0^{-1/D}+o(k):=r.
\end{equation}

Thus we deduce that $R'_i(E)\le r<a_{ii}=1$ if $r<1$, and $E$ is strictly diagonally dominant. From Ger\v sgorin disc theorem in Lemma \ref{thm:G}, all the eigenvalues of $E$ fall into the interval $[1-r,1+r]$. (Here, we use the fact that $E$ is Hermitian, and its eigenvalues are real.) In particular, $E$ is positive definite. Therefore, there is a unique positive definite Hermitian matrix $F=E^{-1/2}$. See e.g. \cite[Theorem 7.2.6]{HJ}.

We next show that $F$ is also strictly diagonally dominant by letting $r$ be smaller. Define the matrix norm $\VERT A\VERT$ of $A$ by its $l^\infty\to l^\infty$ mapping
$$A:\R^n\to\R^n:\VERT A\VERT =\sup_{y\in\R^n}\frac{\|Ay\|_{l^\infty}}{\|y\|_{l^\infty}}.$$

Then
$$\VERT A\VERT =\max_i\sum_{j=1}^m|a_{ij}|.$$

Notice that $E=I+B$, where $I$ is the identity matrix and $B$ is a matrix with diagonal entries all zero. Then $\VERT B\VERT =\max_iR'_i(E)=r<1$, and therefore we can expand $F=E^{-1/2}$ by power series:
$$F=(I+B)^{-\frac12}=I+\sum_{i=1}^\infty\binom{-\frac12}{i}B^i=I+H,$$
in which 
$$\VERT H\VERT \le\sum_{i=1}^\infty\binom{-\frac12}{i}\VERT B\VERT ^i\le3\sum_{i=1}^\infty r^i\le6r.$$

Thus,
\begin{equation}\label{eq:FRin6r}
1-6r\le F_{ii}=1+H_{ii}\le1+6r,\quad\text{and}\quad R'_i(F)=R'_i(H)\le\VERT H\VERT \le6r.
\end{equation}

This shows that $F=I+H$ is strictly diagonally dominant if we choose $6r\le1/4$, that is, $r\le1/24$. Then
\begin{equation}\label{eq:FRin6r1/4}
\frac34\le F_{ii}\le\frac54,\quad\text{and}\quad R'_i(F)\le\frac14.
\end{equation}

This requires that by \eqref{eq:r},
$$r=(7+1296D)c_0^{-1/D}+o(k)\le\frac{1}{24},$$
and thus,
$$(7+1296D)c_0^{-1/D}\le\frac{1}{25},$$
if $k$ is large enough.

Let $u=Fq$ with $u=(u_1,...,u_n)^t$ and $q=(q_1,...,q_n)^t$, then it is easy to check that $\{u_i\}_{i=1}^m$ is orthogonal:
$$\langle u_i,u_j\rangle=\left\langle\sum_sF_{is}q_s,\sum_tF_{jt}q_t\right\rangle=\sum_{s,t}F_{is}F_{jt}E_{st}=\big(FEF\big)_{ij}=I_{ij}.$$

Furthermore, by \eqref{eq:FRin6r1/4},
\begin{eqnarray}\label{eq:up}
\|u_i\|_p=\left\|\sum_{j=1}^mF_{ij}q_j\right\|_p&\ge&\|F_{ii}q_i\|_p-\sum_{j=1,j\ne i}^m\|F_{ij}q_j\|_p\nonumber\\
&\ge&|F_{ii}|\|q_i\|_p-R'_i(F)\|q_j\|_p\ge\frac12C_pk^{\sigma(p)}
\end{eqnarray}
for $2<p\le6$. This finishes the theorem.

\begin{rmk}[Another approach to evaluate $R'_1(E)$]
The above computation of $R'_1(E)$ is rather tedious. Here, we provide the heuristics of a much simpler estimate, which is also the original evidence for me to believe $R'_1(E)\to0$ as $D\to0$. Compute that by \eqref{eq:delta},
\begin{eqnarray*}
&&R'_1(E)\\
&=&\sum_{l=1}^{\lceil\pi/d\rceil}\sum_{x_j\in S_l}|\langle q_1,q_j\rangle|\\
&\le&7\max_{x_j\in S_1}|\langle q_1,q_j\rangle|+\sum_{l=2}^{\lceil\pi/d\rceil}\frac{36\sin(l-1)\delta}{\delta}\,\max_{x_j\in S_l}|\langle q_1,q_j\rangle|\\
&\le&\frac{36}{\delta^2}\left[\sin\delta\left(\cos\frac\delta2\right)^{2k}\cdot\delta+\sum_{l=2}^{\lceil\pi/d\rceil}\sin(l-1)\delta\,\left(\cos\frac{(l-1)\delta}{2}\right)^{2k}\cdot\delta\right]\\
&\le&144\lfloor D(2k+1)\rfloor\cdot\left[\sin\delta\left(\cos\frac\delta2\right)^{2k}\cdot\delta+\sum_{l=2}^{\lceil\pi/d\rceil}\sin(l-1)\delta\,\left(\cos\frac{(l-1)\delta}{2}\right)^{2k}\cdot\delta\right].
\end{eqnarray*}

One observes that in the bracket, 
\begin{equation}\label{eq:Riemann}
\sin\delta\left(\cos\frac\delta2\right)^{2k}\cdot\delta+\sum_{l=2}^{\lceil\pi/d\rceil}\sin(l-1)\delta\,\left(\cos\frac{(l-1)\delta}{2}\right)^{2k}\cdot\delta
\end{equation}
is indeed a Riemann sum of the function 
$$\sin\phi\left(\cos\frac\phi2\right)^{2k}$$
with respect to the partition $\mathcal P$ in $[0,\pi]$, where in the first interval $[0,\delta]$, we choose the value of the function at the right endpoint $\phi=\delta$; and in other intervals, we choose the value at the left endpoint $\phi=(l-1)\delta$. Thus one can expect that when $k\to\infty$, the summation tends to the integral
\begin{equation}\label{eq:integral}
\int_0^\pi\sin\phi\left(\cos\frac\phi2\right)^{2k}\,d\phi=\frac{2}{k+1}.
\end{equation}

Therefore, back to the summation,
$$R'_1(E)\le\frac{288\lfloor D(2k+1)\rfloor}{k+1}\le576D\to0\quad\text{as}\quad D\to0,$$
which is a weaker statement than \eqref{eq:r} since it decreases polynomially with $D\to0$, while \eqref{eq:r} goes to zero exponentially.

However, to rigorously justify that the Riemann sum \eqref{eq:Riemann} converges to the integral \eqref{eq:integral} as $k\to\infty$ is not trivial, since both the partition $\mathcal P$ and the integrating function involve $k$.
\end{rmk}

Next we give a simple proof of Corollary \ref{cor:largep}. 

\begin{proof}[Proof of Corollary \ref{cor:largep}]

From \cite[\S16]{Z2},
\begin{equation}\label{eq:Gaussian}
|Q_k(\phi,\theta)|=|Ck^\frac14(\sin\phi)^ke^{ik\theta}|\approx k^\frac14e^{-\frac k2w^2},
\end{equation}
in which $w=\pi/2-\phi$ and is small. Therefore, the mass of $Q_k$ concentrates in a tube $G_1^w$ around the equator with width $w\sim k^{-1/2}$ and have Gaussian decay in the transverse direction. Furthermore, one can also see that if $w\approx k^{-1/2}$, then $q_1\gtrsim k^{1/4}$. Thus for $p>6$,
$$\|Q_k\|_{L^p(\s^2)}\ge\|q_1\|_{L^p(G_1^w)}\gtrsim k^{\frac14-\frac{1}{2p}}$$
justifying \eqref{eq:Qpnorm}.

Repeat the computation in \eqref{eq:up}, we have
$$\|u_i\|_p\gtrsim k^{\frac14-\frac{1}{2p}}$$
for $p>6$. In particular, $\|u_i\|_\infty\gtrsim k^{1/4}$.
\end{proof}

\begin{rmk}
Recall that the $L^p$ norm maximizers for $p>6$, the zonal harmonics $Z_k$, have
$$\|Z_k\|_p=\Omega\left(k^{\frac12-\frac2p}\right),$$
from \eqref{eq:Zpnorm}. Therefore, the result we have for $p>6$ in Corollary \ref{cor:largep} is weaker than Problem \ref{pblm:largep}, that is, this positive density orthonormal set of spherical harmonics $\{u_i\}$ does not maximize the $L^p$ norm for $p>6$.

Following the same spirit as in the cases of $2<p\le6$, one can of course try to solve Problem \ref{pblm:largep} using zonal harmonics: We first choose a set of (non-orthogonal) zonal harmonics on the sphere with well separated poles, then modify them such that they are orthonormal and still close to zonal harmonics in $L^p$ norms for $p>6$. However, the main obstacle is that the zonal harmonics decrease away from the concentration points rather slowly comparing to the exponential decreasing of Gaussian beams, then prevents one to evaluate the deleted absolute row sums as we did in this section.
\end{rmk}

\section{Localization of $q_i$ and $u_i$}\label{sec:loc}
The construction of the orthonormal set $\{u_i\}_{i=1}^m$ from the Gaussian beams $\{q_i\}_{i=1}^m$ in Section \ref{sec:main} raises a natural question:
$$\text{What do $u_i$'s look like?}$$

Particularly, do they inherit the localization properties of Gaussian beams around great circles? In this section, we characterize the localization properties of $u_i$'s. Roughly speaking, they are still very close to the Gaussian beams in terms of mass localization. In the following discussion, we focus on $u_1$ and $q_1$, and assume that the pole of $q_1$ is the north pole without loss of generality.

\subsection*{Localization of the Gaussian beam $q_1$}
We know that the Gaussian beams $q_1$ concentrates around the equator $G_1$. Moreover, \eqref{eq:Gaussian} implies that $\forall\varepsilon>0$, there exists $c>0$ such that mass outside of the tube $G_1^w$ for $w=ck^{-1/2}$ is smaller than $\varepsilon$:
\begin{equation}\label{eq:qloc}
\|q_1\|_{L^2(\s^2\setminus G_1^w)}=1-\|q_1\|_{L^2(G_1^w)}\le\varepsilon.
\end{equation}

\subsection*{Localization of $u_1$ in $G_1^w$ with fixed width $w$ and density $D$} 
In particular, one deduces that from \eqref{eq:qloc}
$$\|q_1\|_{L^2(G_1^w)}\ge\frac12,$$
where $w=ck^{-1/2}$ for some fixed constant $c$. (Strictly speaking, out notation of ``fixed width $w$'' really means that $w$ has fixed dependence on $k$.)

Recall \eqref{eq:FRin6r} and \eqref{eq:r} that
$$1-6r\le F_{ii}\le1+6r,\quad\text{and}\quad R'_i(F)\le6r,$$
where
$$r=(7+1296D)c_0^{-1/D}+o(k)\to0\quad\text{as}\quad D\to0.$$

If we let $6r\le1/4$, then similar to \eqref{eq:up},
\begin{eqnarray*}
\|u_1\|_{L^2(G_1^\varepsilon)}&=&\left\|\sum_{j=1}^mF_{1j}q_j\right\|_{L^2(G_1^\varepsilon)}\\
&\ge&\|F_{11}q_1\|_{L^2(G_1^\varepsilon)}-\sum_{j=1,j\ne i}^m\|F_{1j}q_j\|_{L^2(\s^2)}\\
&\ge&|F_{11}|\|q_1\|_{L^2(G_1^\varepsilon)}-R'_1(F)\|q_j\|_{L^2(\s^2)}\\
&\ge&\frac34\cdot\frac12-\frac14=\frac18.
\end{eqnarray*}

This means that the $L^2$ mass of $u_1$ in the tube $G_1^w$ has a constant lower bound. To achieve a stronger localization statement of $u_1$ as in \eqref{eq:qloc} for $q_1$, we again have to further maneuver the density $D$. 

\subsection*{Localization of $u_1$ outside $G_1^w$ depending on $w$ and $D$} Resume the notation in \eqref{eq:qloc}, then,
\begin{eqnarray*}
\|u_1\|_{L^2(\s^2\setminus G_1^w)}=\left\|\sum_{j=1}^mF_{1j}q_j\right\|_{L^2(\s^2\setminus G_1^w)}&\le&\|F_{11}q_1\|_{L^2(\s^2\setminus G_1^w)}+\sum_{j=1,j\ne i}^m\|F_{1j}q_j\|_{L^2(\s^2)}\\
&\le&|F_{11}|\|q_1\|_{L^2(\s^2\setminus G_1^w)}+R'_1(F)\|q_j\|_{L^2(\s^2)}\\
&\le&(1+6r)\varepsilon+6r\\
&\le&\varepsilon.
\end{eqnarray*}
for $r$ small by choosing the density $D$ small, and $w=ck^{-1/2}$ for some $c$ depending on $\varepsilon$.

\section*{Acknowledgements}
Steve Zelditch told me Problem \ref{pblm:SZ} in \cite{SZ2} when I visited Northwestern University in May 2013, I thank him and the hospitality of Northwestern University. Alex Barnett explained to me the relation between Wigner D-matrix and spherical harmonics; Alan McIntosh suggested me to use $\VERT\cdot\VERT$ norm of matrices in the computation; Andrew Hassell encouraged me to calculate an explicit density bound in Theorem \ref{thm:main}. I thank them for the help throughout the preparation of this paper. I also want to thank Jean Bourgain for sending his unpublished note \cite{B3} to me.

\end{document}